\documentclass[11pt]{amsart}
\usepackage{url,bm,amssymb,amscd,amsmath}
\usepackage{graphics,epsfig,color}

\usepackage[letterpaper,asymmetric,left=1.25in,right=1.25in,top=1.25in,bottom=1.25in,bindingoffset=0.0in]{geometry}

\usepackage{url,bm,amssymb,amscd,amsmath}
\usepackage{graphics,epsfig,color}
\usepackage[curve,color,graph]{xypic}

\usepackage[curve,color,graph]{xypic}

\renewcommand{\tilde}{\widetilde}

\newcommand{\C}{\mathbb{C}}

\newcommand{\CP}{\mathbb{CP}}
\newcommand{\R}{\mathbb R}

\def\H{{\mathcal H}}
\def\X{{\mathcal X}}

\def\W{{\mathcal W}}

\def\buone{{m_1}}
\def\butwo{{m_2}}
\def\buthree{{m_3}}
\def\Lx{{L_x}}
\def\Ly{{L_y}}
\definecolor{greenish}{RGB}{0,120,0}
\definecolor{purplish}{RGB}{100,0,100}
\definecolor{cyanish}{RGB}{0,180,200}

\newtheorem{theorem}{\bf Theorem}[section]
\newtheorem{proposition}[theorem]{\bf Proposition}
\newtheorem{lemma}[theorem]{\bf Lemma}

\newtheorem*{theorem*}{\bf Theorem}
\newtheorem{corollary}[theorem]{\bf Corollary}

\theoremstyle{remark}
\newtheorem{remark}[theorem]{\bf Remark}

\begin{document}

\title{Complex perspective for the projective heat map acting on pentagons}

\author[S. R. Kaschner]{Scott R. Kaschner}
\address{Butler University Department of Mathematics \& Actuarial Science\\
Jordan Hall, Room 270\\
4600 Sunset Ave\\
Indianapolis, IN 46208,
 United States,}
\email{skaschne@butler.edu}
%\urladdr{www.math.arizona.edu/~skaschner}

\begin{author}[R. K. W. Roeder]{Roland K. W. Roeder}
 \address{Department of Mathematical Sciences \\ IUPUI \\ LD Building, Room 224Q\\
402 North Blackford Street\\
Indianapolis, Indiana 46202-3267\\
 United States }
\email{rroeder@math.iupui.edu}
  %%optional:  \curaddr{current address}%%
   %%optional:  \urladdr{website address}%%
\end{author}


\date{\today}

\begin{abstract}
We place Schwartz's work on the real dynamics of the  projective heat map 
$H$ into the complex perspective by computing its first dynamical degree and gleaning some corollaries about the dynamics of $H$.
\end{abstract}

\maketitle

\section{Introduction}

Let $\mathcal{P}_N$ denote the space of projective equivalence classes of
$N$-gons in $\mathbb{RP}^2$.  The projective heat map is a self-mapping
of $\mathcal{P}_N$ that was introduced by R.~Schwartz in the
monograph~\cite{SCHWARTZ_PMM}.  Suppose $AB$, $BC$, and $CD$ are three
consecutive edges of a polygon $P$.  The {\em projective midpoint} of $BC$
(with respect to the polygon $P$) is defined as $S:=\overleftrightarrow{BC} \cap
\overleftrightarrow{QR}$, where $Q = \overleftrightarrow{AB} \cap
\overleftrightarrow{CD}$ and $R = \overleftrightarrow{AC} \cap
\overleftrightarrow{BD}$.  

\begin{figure}[h!]\label{FIG:PROJ_MIDPOINT}
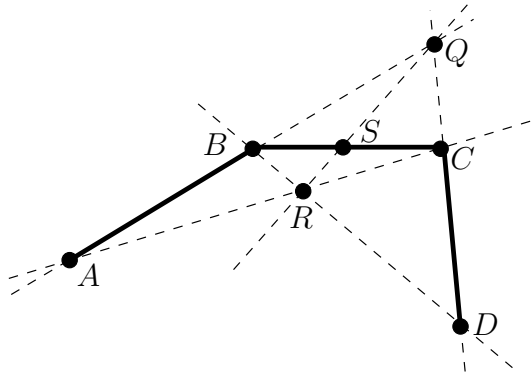
\caption{$S$ is the projective midpoint of $BC$}
\end{figure}

\noindent
The projective midpoint of $BC$ typically does not coincide with the
Euclidean midpoint of $BC$, however the construction is invariant under projective transformations.
Remark also that the projective midpoint may not be defined for certain degenerate configurations of $A,B,C,$ and $D$.

For any $N$-gon $P$ let $H(P)$ be the $N$-gon whose vertices are the projective
midpoints of the edges of $P$.  Because the construction is invariant under
projective transformations, $H$ descends to a mapping $H:
\mathcal{P}_N \dashrightarrow \mathcal{P}_N$ called
the {\em projective heat map}.  (We used a broken arrow
to denote that $H$ may not be defined at certain polygons for which a sequence
of four consecutive vertices $A,B,C$, and $D$ are not in general position.)

In the case of pentagons, the space of projective equivalence classes of polygons $\mathcal{P}_5$ is parameterized by a pair of real numbers $(x,y)$ called the {\em flag invariants} of the (equivalence class) of the polygon; see \cite[Section 3.6]{SCHWARTZ_PMM}.  In these flag coordinates, the projective
heat mapping becomes a rational mapping $H: \mathbb{R}^2 \dashrightarrow \mathbb{R}^2$ given by $(x',y') = H(x,y)$, where
\begin{align}\label{EQN:DEF_MAP}
x' &= {\frac { \left( x{y}^{2}+2\,xy-3 \right)  \left( {x}^{2}{y}^{2}-6\,xy-
x+6 \right) }{ \left( x{y}^{2}+4\,xy+x-y-5 \right)  \left( {x}^{2}{y}^
{2}-6\,xy-y+6 \right) }}, \quad \mbox{and} \\
y' &= {\frac { \left( {x}^{2}y+2\,xy-3 \right)  \left( {x}^{2}{y}^{2}-6\,xy-
y+6 \right) }{ \left( {x}^{2}y+4\,xy-x+y-5 \right)  \left( {x}^{2}{y}^
{2}-6\,xy-x+6 \right) }}. \nonumber
\end{align}
This mapping has an obvious symmetry under the reflection $R(x,y) = (y,x)$.  It
also has a less-obvious symmetry under an action of the dihedral group $D_5$,
corresponding to relabeling the vertices of the pentagon under rotations and reflections.
Expressed as a group of birational mappings of the flag coordinates  $(x,y) \in \R^2$ this action of $D_5$ is
called the Gauss Group $\Gamma$; see \cite[Section 3.8]{SCHWARTZ_PMM}.

In the monograph \cite{SCHWARTZ_PMM} Schwartz uses computer-assisted proofs to provide a nearly complete description of the dynamics
of $H: \mathbb{R}^2 \dashrightarrow \mathbb{R}^2$.  Highlights of his work include proofs that
\begin{enumerate}
\item Almost any projective equivalence class of a pentagon has orbit under $H$ converging to the class of the regular pentagon \cite[Corollary 1.7]{SCHWARTZ_PMM}.  (In flag coordinates, this attracting fixed point is represented by $(\phi^{-1},\phi^{-1})$, where $\phi$ is the golden ratio.)
\item There is a repelling invariant Cantor set $\mathcal{JC}$ for $H$ on which the dynamics of $H$ is conjugate to the one-sided shift $\sigma\colon\Sigma_6\rightarrow\Sigma_6$ on six symbols \cite[Theorem 1.5]{SCHWARTZ_PMM}.
\item $H$ does not have an invariant rational fibration, i.e., there is no nontrivial pair $(h,f)$ of rational functions $f\colon\mathbb R^2\dashrightarrow\mathbb R$ and $h\colon\mathbb R\dashrightarrow\mathbb R$ such that $f\circ H=h\circ f$ \cite[Section 15.7]{SCHWARTZ_PMM}.  This can be interpreted as saying the dynamics of $H$ are truly two-dimensional.
\end{enumerate}
Moreover, much of Schwartz's work is dedicated to giving a 
topological description of the ``Julia set'' of $H$ defined by $\mathcal{J} :=
\mathbb{R}^2 \setminus \W^s((\phi^{-1},\phi^{-1}))$.  A computer generated image
of this Julia set is shown in Figure \ref{FIG_JULIA}.

\begin{figure}
\includegraphics[scale=0.4]{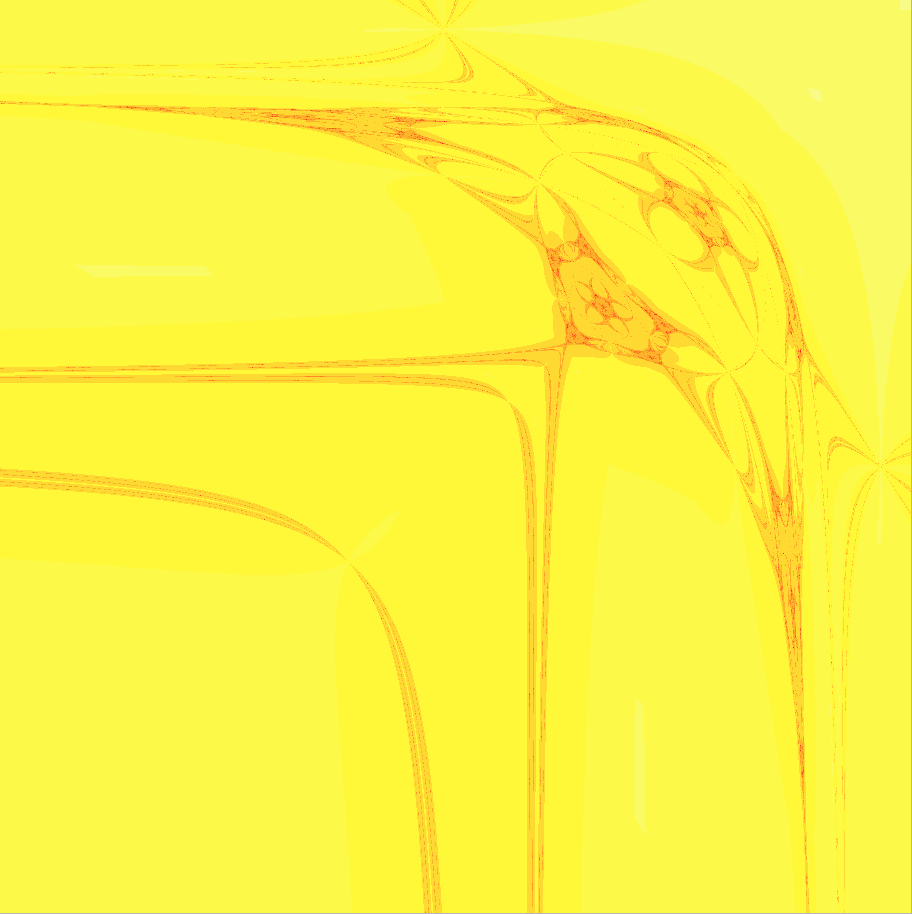}
\caption{\label{FIG_JULIA} The region $(x,y) \in [-9,\phi^{-1}]^2$, where $\phi$ is the golden ratio.  Points in the basin of attraction $\W^s((\phi^{-1},\phi^{-1}))$ of the attracting fixed point $(\phi^{-1},\phi^{-1})$ 
are colored yellow.  The ``Julia Set'' $\mathcal{J}:=\mathbb{R}^2 \setminus \W^s((\phi^{-1},\phi^{-1}))$ is visible in red.}
\end{figure}

Schwartz often 
extends $H$ to the compactification $\mathbb{RP}^1 \times \mathbb{RP}^1$ and sometimes to the
surface $S_\mathbb{R}$ obtained by blowing up $\mathbb{RP}^1 \times \mathbb{RP}^1$ at the
three points $(1,1)$, $(\infty,0)$, and $(0,\infty)$.  The surface $S_\mathbb{R}$ is
well-adapted to the symmetries of $H$ under the Gauss Group $\Gamma$,  since $\Gamma$ 
acts on $S_\mathbb{R}$ by diffeomorphisms; See \cite[Section 10.5]{SCHWARTZ_PMM}.

\vspace{0.1in}
The formula for $H$ naturally extends to a rational map $H: \CP^1 \times
\CP^1 \dashrightarrow \CP^1 \times \CP^1$ or even as a rational self-map $H: S
\dashrightarrow S$, where $S$ is the blow-up of  $\CP^1 \times \CP^1$ at
$(1,1)$, $(\infty,0)$, and $(0,\infty)$.  Starting in the 1990s, Briend-Duval,
Fornaess-Sibony, Hubbard-Oberst-Vorth, Hubbard-Papadapol, Ueda, and many others
used powerful tools from complex analysis and algebraic geometry to prove
strong results about the dynamics of such mappings.  {\em The purpose of this note is to
place Schwartz's work on the real dynamics of the  projective heat mapping
$H$ into this complex perspective and to glean some corollaries both about the dynamics of $H$.}
\vspace{0.1in}

The coarsest invariants of the dynamics of a rational self-mapping $f: X \dashrightarrow X$ of a (complex) projective surface $X$ are the dynamical degrees 
$\lambda_1(f)$ and $\lambda_2(f)$.  The first dynamical degree $\lambda_1(f)$ is defined by
\begin{align*}
\lambda_1(f) := \lim_{n \rightarrow \infty} \left\|(f^n)^* : {\rm H}^{1,1}(X;\C) \rightarrow {\rm H}^{1,1}(X;\C)\right\|^{1/n},
\end{align*}
where ${\rm H}^{1,1}(X;\mathbb{C})$ denotes the Dolbeault cohomology of bidegree
$(1,1)$.  (In the case considered in this note, the reader can replace it with
singular cohomology ${\mathrm H}^2(X;\mathbb{C})$.) There are many subtleties
to this definition, including the fact that the rational map $f$ may not be
continuous at a finite set of indeterminate points.  Nevertheless, there is
still a well-defined notion of pullback on cohomology.  The caveat is that,
unlike for continuous mappings, this pullback may not be functorial---one may
have $(f^n)^* \neq (f^*)^n$.  This often makes calculation of $\lambda_1(f)$
rather technical.

Since we are working in complex dimension two, the second dynamical degree is
much simpler: $\lambda_2(f) = {\rm deg}_{\rm top}(f)$ is the the number of
preimages of a generic point under $f$.  We refer the reader to
\cite{BEDFORD_DYN_DEG,FRIEDLAND,RUSS_SHIFF,DILLER_FAVRE,DS_BOUND,DS_REGULARIZATION}
for further background on dynamical degrees.  The facts we will need about them
will be summarized in Section \ref{SEC_DYN_DEG}.

The two cases $\lambda_2(f) > \lambda_1(f)$ and $\lambda_1(f) > \lambda_2(f)$
correspond to quite different types of dynamics.  Building on work of
Briend-Duval \cite{briend1999exposants,briend2001deux}, Guedj and
Dinh-Truong-Nguyen  \cite{GUEDJ,DNT_degree} proved that the case $\lambda_2(f)
> \lambda_1(f)$ corresponds to ``predominantly repelling dynamics.''
Meanwhile, building from work of Bedford-Lyubich-Smillie, Diller-Dujardin-Guedj
\cite{DDG} proved that the case $\lambda_1(f) > \lambda_2(f)$ corresponds to
``predominantly saddle-type dynamics.''  We refer the reader to the cited papers
for details.

It was proved by Schwartz \cite[Section 8]{SCHWARTZ_PMM} that $\lambda_2(H) = 6$.  Our main technical result is

\begin{theorem}\label{THM:MAIN} $\lambda_1(H) = 4$.
\end{theorem}

Therefore, the projective heat map $H$ has ``large topological degree''
and the work of Guedj and Dinh-Truong-Nguyen  \cite{GUEDJ,DNT_degree} implies
that $H: \CP^1 \times \CP^1 \dashrightarrow  \CP^1 \times \CP^1$ has a unique
measure $\mu$ of maximal entropy, whose entropy is equal to $\log \lambda_2(H)
= \log 6$.  It also provides a lower bound on the Lyapunov exponents $\chi_{1,2}$ of $\mu$
by 
\begin{align*}
\chi_{1,2}  \geq \frac{1}{2} \log \frac{\lambda_2(H)}{\lambda_1(H)}  = \frac{1}{2} \log \frac{3}{2}.
\end{align*}
Meanwhile, the Bernoulli measure on $\Sigma_6$ is an invariant measure of
entropy $\log 6$ for the one-sided full shift on six symbols $\sigma: \Sigma_6 \rightarrow \Sigma_6$.   Pulling it back
under the the conjugacy 
between $H: \mathcal{JC} \rightarrow \mathcal{JC}$ and $\sigma: \Sigma_6
\rightarrow \Sigma_6$ results in an invariant measure $\beta$ for $H:
\mathcal{JC} \rightarrow \mathcal{JC}$ which also has entropy $\log 6$.  We
will call $\beta$ the ``Bernoulli measure on $\mathcal{JC}$.''  We conclude that:

\begin{corollary}The Bernoulli measure $\beta$ on $\mathcal{JC} \subset \mathbb{RP}^1 \times \mathbb{RP}^1$ is the unique
measure of maximal entropy for $H: \CP^1 \times \CP^1 \dashrightarrow \CP^1
\times \CP^1$.  Its Lyapunov exponents are bounded from below by $\frac{1}{2} \log
\frac{3}{2}$.  \end{corollary}

The dynamical degrees of rational maps having an invariant fibration were
studied by Dinh-Nguyen and Dinh-Nguyen-Truong in \cite{NT,DNT_fibration}.
Their work implies that if a rational self-map of a surface has an invariant
fibration, then $\lambda_1(f)$ divides $\lambda_2(f)$; see Lemma
\ref{LEM_DYN_DEG_FIBRATION}, below.  We therefore also conclude that:

\begin{corollary}\label{COR_NO_FIBRATION}The projective heat map $H: \CP^1 \times \CP^1
\dashrightarrow \CP^1 \times \CP^1$ has no invariant fibration.
\end{corollary}

\noindent
Remark that if $H$ preserved a real fibration $(f,g)$, their complexifications would serve as an invariant complex fibration for $H$. In particular, $H: \mathbb{RP}^1 \times \mathbb{RP}^1 \dashrightarrow \mathbb{RP}^1 \times \mathbb{RP}^1$ does not have an invariant fibration, thus giving a cohomological re-proof of one of Schwartz's results.

\begin{remark}
Computation of the first dynamical degree of a planar rational map is relatively standard for specialists in several variable complex dynamics.  However, the purpose of this note is to relate the dynamical consequences of this calculation
to Schwartz's work.  In order to make this paper more accessible to readers outside the field of several variables complex dynamics, we have intentionally included many details that otherwise might have been omitted.
\end{remark}

Section \ref{SEC_DYN_DEG} contains background material on dynamical degrees for
rational maps of surfaces and how they can be used to rule out invariant
fibrations.  Section \ref{SEC_AS} and \ref{SEC_PULLBACK} are devoted to the
calculation that $\lambda_1(H) = 4$, i.e. to proving Theorem \ref{THM:MAIN}.
There are many interesting further questions about the complex dynamics of the
projective heat map.  Therefore, we conclude in Section \ref{SEC_QUESTIONS} by
describing a couple of interesting directions for further study.

\vspace{0.1in}

\noindent
{\bf Acknowledgments:} We thank the referee for several helpful comments and Joseph Silverman for his comments on polarized dynamical systems. The second author was supported by NSF grant DMS-1348589.

\section{Background on dynamical degrees and algebraic stability}
\label{SEC_DYN_DEG}

In this section we will give the background needed for computing
$\lambda_1(H)$.  Let us focus on a somewhat restricted context suitable for
this paper.    Suppose $X$ and $Y$ are complex projective algebraic surfaces
(e.g. a blow-up of $\CP^1 \times \CP^1$ at finitely many points) and $f: X
\dashrightarrow Y$ is a rational map.  We denote the indeterminacy locus of $f$
by $I(f)$, which is a finite set of points.  Throughout the section we will
suppose $f$ is {\em dominant}, meaning that $f(X \setminus I(f))$ is not
contained in an algebraic hypersurface of $Y$.

\subsection{Pullback on cohomology}

It is a well-known fact \cite[Ch. IV, \S 3.3]{SHAF} there is a finite sequence
of blow-ups $\pi: \widetilde{X} \rightarrow X$ so that $f$ lifts to a holomorphic
map $\tilde{f}: \widetilde{X} \rightarrow Y$, making the following diagram
commute
\begin{eqnarray}\label{RESOLUTION}
\xymatrix{
\widetilde{X} \ar[d]^\pi \ar[dr]^{\tilde{f}} & \\
X  \ar @{-->}[r]^{f} & Y,
}
\end{eqnarray}
wherever $f \circ \pi$ is defined.  Each of these blow-ups is done over a point of $I(f)$.

One uses (\ref{RESOLUTION}) to define $f^*: {\rm H}^{1,1}(Y;\C) \rightarrow {\rm H}^{1,1}(X;\C)$ by
\begin{align}\label{DEF_PULLBACK}
f^*(\alpha) := \pi_* (\tilde{f}^* \alpha)
\end{align}
for any $\alpha \in {\rm H}^{1,1}(Y;\C)$.  Here, $\pi_*: {\rm H}^{1,1}(\tilde{X};\C)
\rightarrow {\rm H}^{1,1}(X;\C)$ is defined by $$\pi_* := {\rm PD}_{X}^{-1} \circ
\pi_\# \circ {\rm PD}_{\tilde{X}},$$ with $\pi_\#$ denoting the push forward on
homology and ${\rm PD}$ denoting the Poincar\'e duality isomorphism.  The
definition of the pullback $f^*$ is well-defined, independent of the choice of
resolution of indeterminacy (\ref{RESOLUTION}); see, for example,
\cite[Lemma~3.1]{ROEDER_FUNC}.

If $C \subset X$ is an irreducible algebraic curve, then, since $C$ has real dimension two and is singular at only
finitely many points, it has a well-defined fundamental homology class $\{C\}$
and cohomology class $[C] = {\rm PD}_X^{-1} (\{C\})$.  From a more
sophisticated point of view, $C$ defines a locally principal divisor $(C)$ and
$[C]$ is its Chern class; \cite{GH}.  The definition
of $f^*$ simplifies in this case:

\begin{lemma} \label{LEM:PULLBACK_ALG_CURVE}
 Let $f: X \dashrightarrow Y$ be a dominant rational map between to projective
surfaces.  Suppose $C \subset Y$ is an irreducible algebraic curve.  Then,
\begin{align}\label{PULLBACK_OF_ALGEBRAIC_CURVES}
f^* [C] = \sum_{\substack{D \subset f^{-1}(C) \\ \mbox{irreducible}}} m_D [D],
\end{align}
where $f^{-1}(C) = \overline{{\left(f|_{X \setminus I(f)}\right)^{-1} C}}$ and the
multiplicity $m_D$ is the order of vanishing of $\psi \circ f$ at any smooth
point $p \in D \setminus I(f)$, with $\psi$ being a local defining equation for $C$
at $f(p)$ (chosen to vanish to order $1$ at smooth points of $C$).
\end{lemma}

\begin{proof}
For a holomorphic map $h: X \rightarrow Y$ the pullback
action on locally principal divisors and on cohomology are related by the commutative diagram:
\begin{eqnarray}\label{CHERN}
\xymatrix{
{\rm H}^1(Y,\mathcal{O}^*) \ar[d]^c \ar[r]^{h^*} & {\rm H}^1(X,\mathcal{O}^*) \ar[d]^c \\
{\rm H}^1(Y;\C)  \ar[r]^{h^*} & {\rm H}^1(X;\C),
}
\end{eqnarray}
where $\mathcal{O}^*$ denotes the sheaf of germs of non-vanishing holomorphic
functions and the vertical arrows denote taking the Chern class, see \cite[p.
139]{GH}.  Since (\ref{PULLBACK_OF_ALGEBRAIC_CURVES}) is a special case of the
definition of how one pulls back locally principal divisors, this diagram
justifies (\ref{PULLBACK_OF_ALGEBRAIC_CURVES}) in the case of a holomorphic map.

Now, suppose $f: X \dashrightarrow Y$ is rational, that we have the resolution
of indeterminacy $\tilde{f}: \tilde{X} \rightarrow Y$ as in (\ref{RESOLUTION}),
and let $C \subset Y$ be an irreducible algebraic curve.  By the discussion in
the previous paragraph, (\ref{PULLBACK_OF_ALGEBRAIC_CURVES}) applies to
$\tilde{f}: \tilde{X} \rightarrow Y$ giving
\begin{align}\label{TILDE_PULLBACK0}
\tilde{f}^*[C] = \sum_{\substack{A \subset \tilde{f}^{-1}(C) \\ \mbox{irreducible}}} m_A [A]
\end{align}
with $m_A$ being the order of vanishing of $\psi \circ \tilde{f}$ at any smooth point of $A$.
Therefore, we have
\begin{align}\label{TILDE_PULLBACK}
f^* [C] = \pi_* \tilde{f}^*[C] = \sum_{\substack{A \subset \tilde{f}^{-1}(C) \\ \mbox{irreducible}}} m_A \pi_* [A].
\end{align}
If $A$ is an irreducible component of $\tilde{f}^{-1}(C)$ with $\pi(A)$ a
single point, then by the definition of push forward on homology we have $\pi_*
[A] = 0$.  Otherwise, $\pi(A)$ is an irreducible algebraic curve in $X$.  In
this case, $A$ intersects the exceptional divisors of $\pi$ in finitely many
points, so that $\pi: A \rightarrow \pi(A)$ is one-to-one away from finitely
many points and hence $\pi_*[A] = [\pi(A)]$.  We conclude that
\begin{align}\label{TILDE_PULLBACK2}
f^* [C] = \sum_{\substack{A \subset \tilde{f}^{-1}(C) \, \,  \mbox{irreducible} \\ \mbox{$\pi(A)$ not a point}}} m_A [\pi(A)].
\end{align}
Suppose $p$ is any smooth point of $\pi(A) \setminus I(f)$.  Then, $\pi$ is a
biholomorphic map from a neighborhood $U$ of $\pi^{-1}(p)$ to a neighborhood
$V$ of $p$.  Therefore, if $\psi$ is a local defining equation for $C$ in a
neighborhood of $f(p) = \tilde{f}(\pi^{-1}(p))$, then the order of vanishing of $\psi
\circ f = \psi \circ \tilde{f} \circ \left(\pi|_U\right)^{-1}$ at $p$ is the same as that of
$\psi \circ \tilde{f}$ at $\pi^{-1}(p)$.  This allows us to replace the $m_A$
in (\ref{TILDE_PULLBACK2}) with $m_{\pi(A)}$.  We obtain
\begin{align}\label{TILDE_PULLBACK3}
f^* [C] = \sum_{\substack{D \subset \pi(\tilde{f}^{-1}(C)) \,\, \mbox{irreducible} \\ \mbox{$D$ not a point}}} m_D [D].
\end{align}
Finally, note that each of the blow-ups done to achieve the resolution of indeterminacy (\ref{RESOLUTION}) occurs over a point of $I(f)$.
Therefore, 
\begin{align*}
f^{-1}(C) = \overline{{\left(f|_{X \setminus I(f)}\right)^{-1} C}} \quad \mbox{and} \quad \pi \circ \tilde{f}^{-1}(C)
\end{align*}
differ by at most a subset of the finite set $I(f)$, giving a bijection between
irreducible components of $f^{-1}(C)$ and irreducible components of
$\pi(\tilde{f}^{-1}(C))$ that are not a point.  This gives
(\ref{PULLBACK_OF_ALGEBRAIC_CURVES})~from~(\ref{TILDE_PULLBACK3}).

\end{proof}

\subsection{Algebraic stability and strategy for computing $\lambda_1$}\label{SEC:STRATEGY}

\begin{proposition}\label{PROP:DF}{\rm (Diller-Favre \cite[Thm. 1.14]{DILLER_FAVRE}})
Let $f: X \dashrightarrow X$ by a rational self-map of a projective surface.
Then, $(f^n)^* = (f^*)^n$ for all $n \geq 1$ if and only if
there is no curve $C$ and no iterate $m$ such that $f^m(C \setminus I_{f^m})
\subset I_f$.
\end{proposition}

When either of the two equivalent conditions stated in Proposition
\ref{PROP:DF} hold, we will call $f: X \dashrightarrow X$ {\em algebraically
stable}.   (Readers who prefer to see a proof of Proposition \ref{PROP:DF} that is based on algebraic geometry instead of analysis 
can refer to \cite[Prop. 1.4]{ROEDER_FUNC}.)

If $C$ is a curve such that $f(C \setminus I_f)$ is a point, we will say that
$C$ is {\em collapsed by $f$}.   Since $I_f$ is a finite set of points, the
second equivalent condition for algebraic stability (from Proposition \ref{PROP:DF}) asserts that there is
no curve $C$ that is collapsed by an iterate of $f$ into $I_f$.
In particular, if $f$ does not collapse any curve, then $f$ is algebraically stable.

\begin{proposition}\label{PROP_CONGUGACY}{\rm (Dinh-Sibony \cite[Cor. 7]{DS_BOUND})} Suppose $f: X \dashrightarrow X$ and $g: Y \dashrightarrow Y$ are rational maps of projective surfaces that are conjugate by means of a birational map $\pi: X \dashrightarrow Y$.  Then, $\lambda_1(f) = \lambda_1(g)$.
\end{proposition}

\noindent (Proposition \ref{PROP_CONGUGACY} actually holds
considerably greater generality, including in arbitrary
dimensions, but we will only need the simplest form, as stated here.)

Based on these two propositions, there is a clear strategy for computing the
first dynamical degree of a rational map $f: X \dashrightarrow X$.  One
should try to do a sequence of blow-ups $\pi: \tilde{X} \rightarrow X$
in order to make the lifted map $\tilde{f}: \tilde{X} \dashrightarrow
\tilde{X}$ satisfy the second equivalent condition from Proposition
\ref{PROP:DF}, and hence be algebraically stable.   Then
\begin{align}\label{EQN:STRATEGY}
\lambda_1(f) = \lambda_1(\tilde{f}) = {\rm spectral\, radius}\left(\tilde{f}^*: {\rm H}^{1,1}(\tilde{X}) \rightarrow {\rm H}^{1,1}(\tilde{X})\right),
\end{align}
where the first equality follows from applying Proposition \ref{PROP_CONGUGACY}
to the conjugacy $\pi$ between $f$ and $\tilde{f}$, and the second equality
follows from Proposition \ref{PROP:DF}.   Finally, if the cohomology
${\rm H}^{1,1}(\tilde{X})$ is generated by the fundamental classes of algebraic
curves, one can use Lemma~\ref{LEM:PULLBACK_ALG_CURVE} to compute the
spectral radius on the right hand side of (\ref{EQN:STRATEGY}).

\subsection{Ruling out invariant fibrations by means of dynamical degrees}

\begin{lemma}\label{LEM_DYN_DEG_FIBRATION} Suppose $f: X \dashrightarrow X$ is a dominant rational map of a projective surface that
has an invariant fibration.  Then, $\lambda_1(f)$ divides $\lambda_2(f)$.
\end{lemma}

\begin{proof} 
Suppose that $Y$ is a one-dimensional projective curve and that
$f: X \dashrightarrow X$ is semi-conjugate to a rational
map $g: Y \dashrightarrow Y$, by a rational mapping $\pi: X \dashrightarrow Y$.  The formula from
\cite{NT,DNT_fibration} gives that: 
\begin{align*}
\lambda_1(f) &= {\rm max}\Big(\lambda_1(g) \lambda_0(f|\pi), \lambda_0(g) \lambda_1(f|\pi)\Big), \quad \mbox{and} \\
\lambda_2(f) &= \lambda_1(g) \lambda_1(f|\pi).
\end{align*}
Here, the fiber-wise dynamical degrees are defined as
\begin{align}\label{DEF_FIBERWISE}
\lambda_i(f|\pi) := \lim_{n \rightarrow \infty} \|(f^n)^*(\omega_X^i) \wedge \pi^*(\omega_Y)\|^{1/n}
\end{align}
for $i=0,1$, where $\omega_X$ and $\omega_Y$ are the Fubini-Study forms on $X$
and $Y$ respectively and $\| \cdot \|$ denotes the mass of a current.  (The exponent on $\omega_Y$ is $1$ since ${\rm dim}(Y) = 1$.)  It is always true that the $0$-th dynamical degree of a
mapping is $1$, so that $\lambda_0(g)=1$.  (This is why we didn't mention
$0$-th dynamical degrees in the introduction).  Meanwhile, it follows
immediately from (\ref{DEF_FIBERWISE}) that $\lambda_0(f|\pi) = 1$.  Therefore,
$\lambda_1(f) = \max\left(\lambda_1(g),\lambda_1(f|\pi)\right)$.  In either case, the
result follows from $\lambda_2(f) = \lambda_1(g) \lambda_1(f|\pi)$.
\end{proof}

\noindent
Corollary \ref{COR_NO_FIBRATION} follows.

\section{Algebraic stability on a suitable blow-up of $\CP^1\times \CP^1$.}\label{SEC_AS}
We use multi-homogeneous coordinates $([X:U],[Y:V])$ on $\CP^1 \times \CP^1$, where the affine coordinates from (\ref{EQN:DEF_MAP}) correspond to $(x,y) = (X/U,Y/V)$.  There are three other ``standard'' choices of local coordinates on $\CP^1\times \CP^1$ given by
\begin{align*}
(u,y) := (U/X,Y/V), \qquad (x,v) := (X/U,V/Y), \qquad \mbox{and} \qquad (u,v) :=(U/X,V/Y).
\end{align*}
In the $(x,y)$ local coordinates, the critical set of $H: \mathbb{C}^2 \dashrightarrow \mathbb{C}^2$ consists of the following five irreducible curves:
{\footnotesize
\begin{align*}
C_1 &= \{xy-1 = 0\}, \\
C_2 &= \{2xy+x+y-4 =0\}, \\
C_3 &= \{x^2y^2-6xy-y+6 = 0\}, \\
C_4 &= \{x^2y^2-6xy-x+6 = 0\}, \mbox{and} \\
C_5 &= \{x^6y^6-10x^5y^5-x^6y^3+2x^5y^4+2x^4y^5-x^3y^6-4x^5y^3+39x^4y^4-4x^3y^5+3x^5y^2-12x^4y^3\\
&-12x^3y^4+3x^2y^5+10x^4y^2-47x^3y^3+10x^2y^4-3x^4y^2+22x^3y^2+22x^2y^3-3xy^4-12x^3y-2x^2y^2\\
&-12xy^3-6x^2y-6xy^2+9x^2+21xy+9y^2-9x-9y = 0\}.
\end{align*}
}

\begin{lemma} The only curves collapsed by $H:\mathbb{CP}^1\times\mathbb{CP}^1\dashrightarrow\mathbb{CP}^1\times\mathbb{CP}^1$
are $C_1,\ldots,C_4$.  
\begin{itemize}
\item $C_1$ and $C_2$ are collapsed by $H$ to $p_1 := \{(x,y) = (1,1)\}$, 
\item $C_3$  is collapsed by $H$ to $p_2:=\{(x,y)=(\infty,0)\} = \{(u,y) = (0,0)\}$, and
\item $C_4$ is collapsed by $H$   to $p_3 := \{(x,y)=(0,\infty)\} =  \{(x,v) = (0,0)\}$.
\end{itemize}
The curve $C_5$ is not collapsed by $H$.
\end{lemma}

\begin{proof}
Any curve that is collapsed by $H$ is in the critical set, and it is also easy to verify in $(u,v)$ coordinates that the lines at infinity, $\{u=0\}$ and $\{v=0\}$, are not collapsed.  Thus, we need only consider $C_1,\ldots,C_5$.

The variable $x$ occurs at most linearly in the defining equations for
$C_1$ and $C_2$.  One can therefore solve for $x$ as a function of $y$.  Substituting into the defining equation for $H$, (\ref{EQN:DEF_MAP}) yields that each curve $C_1$ and $C_2$ collapses to $p_1$.

To see that $C_3$ is collapsed to $p_2$, we use $(x,y)$ coordinates in the domain and 
$(u,y) = (1/x,y)$ in codomain:
\begin{align}
u'=1/x' &= {\frac { \left( x{y}^{2}+4\,xy+x-y-5 \right)  \left( {x}^{2}{y}^
{2}-6\,xy-y+6 \right) }{ \left( x{y}^{2}+2\,xy-3 \right)  \left( {x}^{2}{y}^{2}-6\,xy-
x+6 \right) }}, \quad \mbox{and} \\
y' &= {\frac { \left( {x}^{2}y+2\,xy-3 \right)  \left( {x}^{2}{y}^{2}-6\,xy-
y+6 \right) }{ \left( {x}^{2}y+4\,xy-x+y-5 \right)  \left( {x}^{2}{y}^
{2}-6\,xy-x+6 \right) }}. \nonumber
\end{align}
Since the defining equation of $C_3$ appears in the numerator of both $u'$ and
$y'$, it follows that $C_3$ collapses to $(u,y) = (0,0)$. 
The fact that $C_4$ is collapsed by $H$ to $p_3$ follows by symmetry under $R$, or by using a similar calculation.

Meanwhile, $(x,y) = (0,0)$ and $(x,y) = (0,1)$ are on $C_5$ and they map by $H$ to different points. %(3/5,3/5) and (9/14,7/12)
\end{proof}

Let $\mathcal{X}$ be the blow-up of $\CP^1\times\CP^1$ at $p_1$, $p_2$, and $p_3$, let
$E_1, E_2,$ and $E_3$ be the resulting exceptional divisors, respectively, and
let $\pi\colon \mathcal{X} \rightarrow \CP^1\times\CP^1$ be the canonical projection.  We
let $\H: \mathcal{X} \dashrightarrow \mathcal{X}$ be the lift of
$H$ to this space such that the following diagram commutes:

\begin{align}
\label{EQN:CONJ}
\xymatrix{
\X \ar @{-->}[r]^\H  \ar[d]^\pi  & \X \ar[d]^\pi  \\
\CP^1\times\CP^1  \ar @{-->}[r]^H & \CP^1\times\CP^1.
}
\end{align}

For any algebraic curve $C \subset \mathbb{CP}^1 \times \mathbb{CP}^1$ 
let 
\begin{align*}
\widetilde{C} = \overline{\pi^{-1}(C \setminus \{p_1,p_2,p_3\})}
\end{align*}
denote the {\em proper transform} of $C$ under $\pi$.

The local coordinates $(x,y), (u,y), (x,v)$ and $(u,v)$ on
$\mathbb{P}^1 \times \mathbb{P}^1$ continue to serve as local
coordinates on $\mathcal{X}$ away from the exceptional divisors $E_1, E_2$, and
$E_3$.  We now set up some standard choices of local coordinates on
$\mathcal{X}$ near these exceptional divisors.  Consider the local coordinates
$(a,b) = (x-1,y-1)$ on $\mathbb{P}^1 \times \mathbb{P}^1$ centered at $p_1$.
One can describe a neighborhood of $E_1$ in $\X$ using two systems of local
coordinates.  The local coordinates $(a,\buone)=(a,b/a)$ describe all
directions of approach to $p_1$, except along the line $\{x=1\}$, which
corresponds to $m_1 = \infty$.  (We actually won't need the other system of
coordinates in this paper.)
Similarly, the local coordinates $(u,\butwo) = (u,y/u)$ describe a neighborhood
of all but one point of the exceptional divisor $E_2$, and the local coordinates
$(v,\buthree) = (v,x/v)$ describe a neighborhood of all but one point of
the exceptional divisor $E_3$.  See Figure \ref{FIG_BLOWNUP}.

\begin{figure}
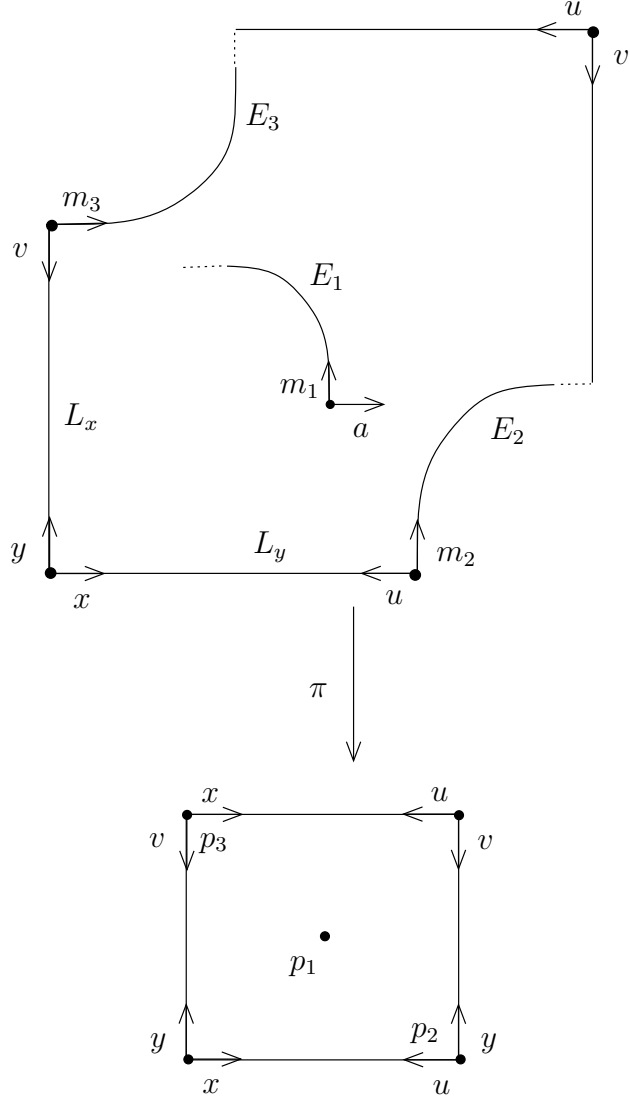
\caption{\label{FIG_BLOWNUP} On the bottom is $\CP^1\times\CP^1$, with the four sets of local coordinates labeled.  On the top is $\mathcal{X}$, the blow-up of $\CP^1\times\CP^1$ at $p_1$, $p_2$, and $p_3$, with the respective exceptional divisors $E_1, E_2,$ and $E_3$.}
\end{figure}

\begin{lemma}\label{LEM:AS}
$\H: \mathcal{X}  \dashrightarrow \mathcal{X}$ does not collapse any curves and is therefore algebraically stable.  
In particular,
\begin{itemize}
\item $\H$ maps $\widetilde{C_1}$ and $\widetilde{C_2}$ onto all of $E_1$
\item  $\H$ maps $\widetilde{C_3}$ onto all of $E_2$, and
\item $\H$ maps $\widetilde{C_4}$ onto all of $E_3$.
\end{itemize}
\end{lemma}

\begin{proof}
If $\mathcal{H}$ collapses a curve $A$ that is not one of the exceptional divisors $E_i$ of $\pi$,
then $\pi(A)$ is a curve that is collapsed by $H$.  Therefore,
since $C_1,\ldots,C_4$ are the only curves
collapsed by $H$, it suffices to check that $\mathcal{H}$
does not collapse $E_1, E_2, E_3, \widetilde{C_1}, \ldots, \widetilde{C_3}$, or $\widetilde{C_4}$.

We will first check that $\H$ does not collapse the exceptional divisors
$E_1,E_2$, and $E_3$.  If we express $\H$ in using $(a,\buone )$ coordinates in the domain
and $(x,y)$ in the codomain, we find
{\footnotesize
\begin{align*}
x' &= {\frac { \left( {a}^{2}{\buone }^{2}+a{\buone }^{2}+4\,a\buone +4\,\buone +3 \right)  \left( {
a}^{3}{\buone }^{2}+2\,{a}^{2}{\buone }^{2}+2\,{a}^{2}\buone +a{\buone }^{2}-2\,a\buone +a-4\,\buone -5
 \right) }{ \left( {a}^{2}{\buone }^{2}+a{\buone }^{2}+6\,a\buone +5\,\buone +6 \right) 
 \left( {a}^{3}{\buone }^{2}+2\,{a}^{2}{\buone }^{2}+2\,{a}^{2}\buone +a{\buone }^{2}-2\,a\buone +a-
5\,\buone -4 \right) }}
 \\
y' &= {\frac { \left( {a}^{2}\buone +4\,a\buone +a+3\,\buone +4 \right)  \left( {a}^{3}{\buone }^{2}
+2\,{a}^{2}{\buone }^{2}+2\,{a}^{2}\buone +a{\buone }^{2}-2\,a\buone +a-5\,\buone -4 \right) }{
 \left( {a}^{2}\buone +6\,a\buone +a+6\,\buone +5 \right)  \left( {a}^{3}{\buone }^{2}+2\,{a}^
{2}{\buone }^{2}+2\,{a}^{2}\buone +a{\buone }^{2}-2\,a\buone +a-4\,\buone -5 \right) }}
\end{align*}
}
The image of $E_1$ is obtained by setting $a = 0$:
%{\footnotesize
\begin{align*}
(x',y') = \left({\frac { \left( 4\,\buone +3 \right)  \left( 4\,\buone +5 \right) }{ \left( 5\,\buone +
6 \right)  \left( 5\,\buone +4 \right) }},{\frac { \left( 3\,\buone +4 \right) 
 \left( 5\,\buone +4 \right) }{ \left( 6\,\buone +5 \right)  \left( 4\,\buone +5
 \right) }}\right).
\end{align*}
%}
Since these functions are not constant functions of $\buone $, we conclude that $\H$ does not collapse~$E_1$. % It will be relevant later that $\H(E_1)$ is not the proper transform of the line $L = \{x+y+1=0\}$.

If we use the local coordinates $(u,\butwo )$ in a neighborhood of generic points of $E_2$ and coordinates $(x,y)$ in the codomain, we find
%{\footnotesize
\begin{align}\label{EQN:MAPPING_NEAR_E2}
x' &= {\frac { \left( {\butwo }^{2}u+2\,\butwo -3 \right)  \left( {\butwo }^{2}u-6\,\butwo u+6\,u-1
 \right) }{ \left( {\butwo }^{2}{u}^{2}-\butwo {u}^{2}+4\,\butwo u-5\,u+1 \right) 
 \left( {\butwo }^{2}-\butwo u-6\,\butwo +6 \right) }}
 \\
y' &= {\frac {u \left( 2\,\butwo u+\butwo -3\,u \right)  \left( {\butwo }^{2}-\butwo u-6\,\butwo +6
 \right) }{ \left( \butwo {u}^{2}+4\,\butwo u+\butwo -5\,u-1 \right)  \left( {\butwo }^{2}u-6
\,\butwo u+6\,u-1 \right) }}. \nonumber
\end{align}
%}
The image of $E_2$ is obtained by setting $u=0$:
\begin{align}\label{EQN:E_2_MAPS_TO_LY}
(x',y') = \left({\frac {3 - 2\,\butwo}{{\butwo }^{2}-6\,\butwo +6}},0\right).
\end{align}
Again, since these functions are not constant functions of $\butwo$, we conclude that $\H$ does not collapse $E_2$. 
The fact that $\H$ does not collapse $E_3$ follows by symmetry under $R$.

It remains to show that the curves $\widetilde C_1$, $\widetilde C_2$, $\widetilde C_3$,
and $\widetilde C_4$ whose projections by $\pi$ were collapsed by $H$ are not
collapsed by $\H$.  We check $\widetilde C_1$ and $\widetilde C_2$ by expressing $\mathcal{H}$ using the local
coordinates $(x,y)$ in the domain and $(a,\buone )$ in the codomain:
\begin{align}\label{E1COORDS}
a' &= -{\frac { \left( 2\,xy+x+y-4 \right)  \left( xy-y-3 \right)  \left( xy-
1 \right) }{ \left( x{y}^{2}+4\,xy+x-y-5 \right)  \left( {x}^{2}{y}^{2
}-6\,xy-y+6 \right) }}
 \\
\buone ' &= {\frac { \left( {x}^{2}{y}^{2}-6\,xy-y+6 \right)  \left( x{y}^{2}+4\,x
y+x-y-5 \right)  \left( xy-x-3 \right) }{ \left( xy-y-3 \right) 
 \left( {x}^{2}y+4\,xy-x+y-5 \right)  \left( {x}^{2}{y}^{2}-6\,xy-x+6
 \right) }}. \nonumber
\end{align}
Recall that $\widetilde C_1=\{xy-1=0\}$, and substituting $y=1/x$ into the equation for $\buone '$, we have
\[\buone'=\frac{x(x+2)}{2x+1},\]
a non-constant function.  Thus, $\buone '$ varies while traversing $\widetilde C_1$, so $\H$ does not collapse $\widetilde C_1$.  

Similarly for $\widetilde C_2=\{2xy+x+y-4=0\}$, substituting $y=\frac{4-x}{2x+1}$ into $\buone '$, we have
\[\buone'={\frac {9({x}^{2}+x+1)}{ \left( {x}^{2}+x+7 \right)  \left( 2\,x+1
 \right) }}.\]
 This is also non-constant, so $\H$ does not collapse  $\widetilde C_2$.
The defining equations for $\widetilde C_1$ and $\widetilde C_2$
appear in the numerator of $a'$ so that they are both mapped by $\mathcal{H}$ to $E_1 = \{a =
0\}$.  Since they are not collapsed by $\mathcal{H}$ and $E_1$ is irreducible,
we conclude that $\widetilde C_1$ and $\widetilde C_2$ are mapped onto all of $E_1$.  

To check that $\widetilde C_3$ is mapped onto all of $E_2$, we use local coordinates $(x,y)$ in the domain and $(u,\butwo )$ in the codomain:
\begin{align*}
u' &= {\frac { \left( x{y}^{2}+4\,xy+x-y-5 \right)  \left( {x}^{2}{y}^{2}-6
\,xy-y+6 \right) }{ \left( x{y}^{2}+2\,xy-3 \right)  \left( {x}^{2}{y}
^{2}-6\,xy-x+6 \right) }}
\\
\butwo' &= {\frac { \left( x{y}^{2}+2\,xy-3 \right)  \left( {x}^{2}y+2\,xy-3
 \right) }{ \left( x{y}^{2}+4\,xy+x-y-5 \right)  \left( {x}^{2}y+4\,xy
-x+y-5 \right) }}.
\end{align*}
%Since the defining equation for $C_3=\{x^{2}{y}^{2}-6\,xy-y+6=0\}$ appears in the numerator of $u'$ and not in the formula for $\butwo$
Since $(x,y) = (0,6)$ and $(x,y) = (1,6)$ are on $\widetilde C_3$ and are mapped by $\mathcal{H}$ to different points on $E_2$, %(0,-9/11) and (0,9/20)
we conclude that $\widetilde C_3$ is not collapsed by $\mathcal{H}$.  Meanwhile, the defining equation for $\widetilde{C_3}$ is a factor of the numerator of the equation for $u'$, so we conclude that $\mathcal{H}$ maps $\widetilde{C_3}$ onto all of $E_2 = \{u=0\}$.
By the symmetry of $\H$ under $R$, or very
similar calculations, one can also show that $\widetilde C_4$ is mapped onto all of
$E_3$.

We conclude that $\mathcal{H}$ does not collapse any curve.  Therefore,
Proposition \ref{PROP:DF} gives that $\H$ is algebraically stable.  
\end{proof}

\section{Pullback on cohomology and computation of $\lambda_1(H)$.}\label{SEC_PULLBACK}

At this point we have completed all but the last step of the strategy for
computing $\lambda_1(H)$ that was presented in Section \ref{SEC:STRATEGY}.  We
have proved in Lemma \ref{LEM:AS} that  $\H : \X \dashrightarrow \X$ is
algebraically stable so we have
\begin{align}\label{EQN:EQUAL_DEGREES}
\lambda_1(H) = \lambda_1(\H) = {\rm spectral\, radius}\left(\H^*: {\rm H}^{1,1}(\X) \rightarrow {\rm H}^{1,1}(\X)\right).
\end{align}
It remains to use Lemma \ref{LEM:PULLBACK_ALG_CURVE} to compute $\H^*: {\rm H}^{1,1}(\X)
\rightarrow {\rm H}^{1,1}(\X)$.   The first step is to
choose a good basis.  It is well-known that
\begin{align*}
{\rm H}^{1,1}(\CP^1\times\CP^1) \cong \mathbb{Z}^2
\end{align*} and is generated by $\left[\Lx\right] = \left[\{X=0\}\right]$ and $\left[\Ly\right] = \left[\{Y=0\}\right]$.  Therefore, by \cite[P. 474]{GH},
\begin{align}\label{EQN:BASIS}
\pi^{\ast}[\Lx], \quad \pi^{\ast}[\Ly], \quad \left[E_1\right], \quad \left[E_2\right], \quad \mbox{and} \quad \left[E_3\right]
\end{align}
is an ordered basis for ${\rm H}^{1,1}(\X)$, where
$\pi^{\ast}[\Lx]$ and $\pi^{\ast}[\Ly]$ are the total transforms of $\Lx$ and $\Ly$ respectively.% under $\pi$. 

\begin{lemma} \label{LEM:PULLBACK}  Let $C_6=\{xy^2+2xy-3=0\}$ and $C_7=\{x^2y+2xy-3=0\}$.  We have
\begin{align}
\H^*\left(\pi^{\ast}[\Lx]\right) &= \left[\widetilde C_6\right]+\left[E_3\right]+\left[\widetilde C_4\right], \label{EQN:PULLBACK_LX} \\
\H^*\left(\pi^{\ast}[\Ly]\right) &= \left[\widetilde C_7\right]+\left[E_2\right]+\left[\widetilde C_3\right], \label{EQN:PULLBACK_LY}\\
\H^*\left[E_1\right] &= \left[\widetilde C_1\right]+\left[\widetilde C_2\right], \label{EQN:PULLBACK_E1} \\
\H^*\left[E_2\right] &= \left[\widetilde{C_3}\right], \quad  \mbox{and}  \label{EQN:PULLBACK_E2}  \\
\H^*\left[E_3\right] &= \left[\widetilde{C_4}\right].  \label{EQN:PULLBACK_E3} 
\end{align}
\end{lemma}

%This is illustrated in Figure \ref{FIG_PULLBACK}.
%
%\begin{figure}
%\input{pullback.pstex_t}
%\caption{\label{FIG_PULLBACK} 
%On the leftt is $\mathcal{X}$, the blow-up of $\CP^1\times\CP^1$ at $p_1$, $p_2$, and $p_3$, with the respective exceptional divisors \textcolor{\colorthree}{$E_1$}, \textcolor{\colorfour}{$E_2$}, and \textcolor{\colorfive}{$E_3$} and the curves \textcolor{\colorone}{$\widetilde{\Lx}$} and \textcolor{\colortwo}{$\widetilde{\Ly}$}.  The colors in the diagram on the left indicate which irreducible curves make up the pullback of the corresponding colored curve on the right.
%}
%\end{figure}

\begin{proof}

We'll start by computing $\H^*\left[E_i\right]$ for each $1\leq i\leq3$.  If we express $\H$ in local coordinates $(x,y)$ in the domain and $(a,\buone)$ in the codomain, as in (\ref{E1COORDS}), $E_1$ is locally
$\{a=0\}$.  We see that the defining equations for both $C_1$ and $C_2$
appear exactly once in the numerator of the expression for $a'$, so that
$\left[\widetilde{C_1}\right]$ and $\left[\widetilde{C_2}\right]$ are each assigned multiplicity one in $\H^*\left[E_1\right]$.  (Had the defining equation of $C_1$, for example, appeared with an exponent larger than
one in the expression for $a'$, the coefficient on homology would correspond
to that exponent.)
  The only additional term in the product is $xy-x-3$.  However, this also appears in the
denominator of the expression for $\buone'$, so the curve $\{xy-x-3=0\}$ is
mapped to a line at infinity.   This proves (\ref{EQN:PULLBACK_E1}).

The only curve mapping by $\H$ to $E_2$ is $\widetilde{C_3}$.  A calculation in local coordinates like above shows that $\left[\widetilde{C_3}\right]$ occurs with multiplicity one in $\H^*\left[E_2\right]$, proving (\ref{EQN:PULLBACK_E2}).   Again, we get (\ref{EQN:PULLBACK_E3}) by symmetry. 

Lastly, we compute $\H^*\left(\pi^{\ast}[\Ly]\right)$ and then obtain $\H^*\left(\pi^{\ast}[\Lx]\right)$ 
by symmetry under $R$.  Note that 
\[\H^*\left(\pi^{\ast}[\Ly]\right)\ =\ \H^*\left(\left[\widetilde\Ly\right]+[E_2]\right)\ =\ \H^*\left[\widetilde\Ly\right]+\left[\widetilde C_3\right],\]
where $\widetilde\Ly$ is the proper transform of $\Ly$.  It remains to compute $\H^*\left[\widetilde\Ly\right]$.  One sees from (\ref{EQN:DEF_MAP}) that the only curves 
mapped by $H$ to $\Ly$ are $C_3$ and $C_7$.  After performing the blowups, we saw in Lemma
\ref{LEM:AS} that $\H$ maps $\widetilde{C_3}$ onto $E_2$.  Therefore, $\widetilde{C_3}$
does not contribute to $\H^*\left[\widetilde{\Ly}\right]$.  We also saw in
(\ref{EQN:E_2_MAPS_TO_LY}) that $\H$ maps $E_2$ onto $\widetilde{\Ly}$.  Thus,
$\H^*\left[\widetilde{\Ly}\right]$ is a sum of $\left[\widetilde{C_7}\right]$ and $\left[E_2\right]$ with suitable
multiplicities.   The defining equation for $C_7$
occurs with multiplicity one in the equation for $y'$ in (\ref{EQN:DEF_MAP}).  Meanwhile, $E_2$ is $\{u=0\}$ in the $(u,m_2)$ coordinates.  Since $u$ 
occurs with multiplicity one in the equation for 
$y'$ in (\ref{EQN:MAPPING_NEAR_E2}),  
we see that $\left[E_2\right]$ is also assigned multiplicity one in
$\H^*\left[\widetilde\Ly\right]$.  This proves (\ref{EQN:PULLBACK_LY}), and 
Equation (\ref{EQN:PULLBACK_LX}) follows immediately under the symmetry $R$.
\end{proof}

\begin{remark}
Some caution is required when computing $\H^{\ast}(\pi^{\ast}[\Ly])$.  Indeed,
\begin{eqnarray*}
\pi^{\ast}\circ H^{\ast}&\neq&\left(\H\circ\pi\right)^{\ast}\ =\ \left(\pi\circ H\right)^{\ast}\ =\ \H^{\ast}\circ\pi^{\ast},
\end{eqnarray*}
where the first inequality comes from the fact that $\pi$ collapses curves into
the indeterminacy points of $H$ and the Diller-Favre criterion (Proposition
\ref{PROP:DF}), the middle equality comes from commutative diagram
(\ref{EQN:CONJ}), and the last equality comes from the fact that $\H$ does
not collapse any curves (Lemma \ref{LEM:AS}) and again the Diller-Favre
criterion.  

More specifically, one can check that $H^{\ast}[\Ly]=[C_3]+[C_7]$,
so
\begin{eqnarray*}
\pi^{\ast}\left(H^{\ast}[\Ly]\right)&=&\pi^{\ast}\left([C_3]+[C_7]\right)\ =\ \left[\widetilde C_3\right]+\left[\widetilde C_7\right]+2[E_1]+2[E_2]+2[E_3]\\
&\neq&\H^{\ast}\left(\pi^{\ast}[\Ly]\right).
\end{eqnarray*}
%However, it was shown in Lemma \ref{LEM:PULLBACK} that
%\begin{eqnarray*}
%\H^{\ast}\left(\pi^{\ast}[\Ly]\right)&=&\left[\widetilde C_3\right]+\left[\widetilde C_7\right]+[E_3].
%\end{eqnarray*}
\end{remark}

We now need to re-express each of the cohomology classes on the right hand
sides of (\ref{EQN:PULLBACK_LX}-\ref{EQN:PULLBACK_E3}) in terms of the ordered
basis (\ref{EQN:BASIS}).  We first express the cohomology classes $\left[C_1\right],\ldots,\left[C_4\right]$,$\left[C_6\right]$ and $\left[C_7\right]$ in the basis $\left[\Lx\right]$ and $\left[\Ly\right]$
for ${\rm H}^{1,1}(\CP^1\times\CP^1)$:

\begin{lemma}\label{LEM:COH1} In ${\rm H}^{1,1}(\CP^1 \times \CP^1)$ we have
\begin{align}
\left[{C_1}\right] &=  \left[\Lx\right] + \left[\Ly\right], \label{EQN_PLAIN1} \\
\left[{C_2}\right] &= \left[{\Lx}\right] + \left[\Ly\right], \label{EQN_PLAIN2} \\
\left[{C_3}\right] &= 2\left[{\Lx}\right] + 2\left[\Ly\right], \label{EQN_PLAIN3} \\
\left[{C_4}\right] &= 2\left[{\Lx}\right] + 2\left[\Ly\right], \label{EQN_PLAIN4} \\
\left[C_6\right]  &=  \left[\Lx\right] + 2\left[\Ly\right], \quad \mbox{and} \label{EQN_PLAIN6} \\
\left[C_7\right] &=  2\left[\Lx\right] + \left[\Ly\right]. \label{EQN_PLAIN7}
\end{align}
\end{lemma}

\begin{proof}
Equations (\ref{EQN_PLAIN1}) - (\ref{EQN_PLAIN7}) follow from the well-known fact that if a curve $C\subset\mathbb{CP}^1\times\mathbb{CP}^1$ is defined in standard affine coordinates by a polynomial $p(x,y)=0$, then 
\[[C]=d_x[L_x]+d_y[L_y],\]
 where $d_x$ is the degree of $p$ with respect to $x$, and $d_y$ is the degree of $p$ with respect to $y$.
\end{proof}

We now lift the results of Lemma \ref{LEM:COH1} to the blown-up space $\X$:

\begin{lemma} \label{LEM:COH2} In ${\rm H}^{1,1}(\X)$ we have
\begin{align}
\left[\widetilde{C_1}\right] &= \pi^{\ast}\left[\Lx\right] + \pi^{\ast}\left[\Ly\right] - \left[E_1\right] - \left[E_2\right] - \left[E_3\right], \label{EQN:FANCY1}\\
\left[\widetilde{C_2}\right] &=\pi^{\ast}\left[\Lx\right] + \pi^{\ast}\left[\Ly\right] - \left[E_1\right],  
\label{EQN:FANCY2} \\
\left[\widetilde{C_3}\right] &= 2\pi^{\ast}\left[\Lx\right] + 2\pi^{\ast}\left[\Ly\right] - \left[E_1\right] - 2\left[E_2\right] - \left[E_3\right],  \label{EQN:FANCY3} \\
\left[\widetilde{C_4}\right] &= 2\pi^{\ast}\left[\Lx\right] + 2\pi^{\ast}\left[\Ly\right] - \left[E_1\right] - \left[E_2\right] - 2\left[E_3\right],  \label{EQN:FANCY4} \\
\left[\widetilde C_6\right]  &= \pi^{\ast}\left[\Lx\right] + 2\pi^{\ast}\left[\Ly\right] - \left[E_1\right] - \left[E_2\right] - \left[E_3\right], \quad \mbox{and}  
\label{EQN:FANCY6} \\
\left[\widetilde C_7\right] &=  2\pi^{\ast}\left[\Lx\right]+ \pi^{\ast}\left[\Ly\right] - \left[E_1\right] - \left[E_2\right] - \left[E_3\right].  \label{EQN:FANCY7}
\end{align}
\end{lemma}

\begin{proof}
Since $C_1$ has multiplicity $1$ at $p_1$, $p_2$, and $p_3$, and $\left[C_1\right]=\left[\Lx\right]+\left[\Ly\right]$ by Lemma \ref{LEM:COH1}, we can pull back both sides of the equation by $\pi$ to see that
\[\left[\widetilde{C_1}\right] + \left[E_1\right] + \left[E_2\right] + \left[E_3\right] =
\pi^{\ast}\left[\Lx\right] + \pi^{\ast}\left[\Ly\right].\]
Equation (\ref{EQN:FANCY1}) follows.  Equations (\ref{EQN:FANCY2}-\ref{EQN:FANCY7}) follow by similar calculations.  
The only subtlety is computing the local multiplicities, which we summarize here:
\begin{align*}
\begin{array}{|c||c|c|c|}
\hline
\mbox{Curve} & \mbox{mult. at $p_1$} &  \mbox{mult. at $p_2$} & \mbox{mult. at $p_3$} \\
\hline
\hline
C_2 & 1 & 0 & 0 \\
\hline
C_3 & 1 & 2 & 1 \\
\hline
C_4 & 1 & 1 & 2 \\
\hline
C_6 & 1 & 1 & 1 \\
\hline
C_7 & 1 & 1 & 1 \\
\hline
\end{array}
\end{align*}
\end{proof}

Lemmas \ref{LEM:PULLBACK} and \ref{LEM:COH2} make it easy to express $\H^*:
{\rm H}^{1,1}(\X) \rightarrow {\rm H}^{1,1}(\X)$ in the ordered basis (\ref{EQN:BASIS}).
For example, 
\begin{align}
\H^*\left[\widetilde\Lx\right] &= \left[\widetilde C_6\right]+\left[E_3\right]+\left[\widetilde C_4\right] 
= 3\pi^{\ast}\left[\Lx\right] + 4\pi^{\ast}\left[\Ly\right] - 2\left[E_1\right] - 2\left[E_2\right] - 2\left[E_3\right].
\end{align}
Computing all of the others, we find
\begin{align}\label{EQN_MATRIX}
\H^* = 
\left[ \begin {array}{ccccc} 3&4&2&2&2
\\ \noalign{\medskip}4&3&2&2&2\\ 
\noalign{\medskip}-2&-2&-2&-1&-1\\ 
\noalign{\medskip}-2&-2&-1&-2&-1\\ 
\noalign{\medskip}-2&-2&-1&-1&-2\end {array} \right] 
%\left[ \begin {array}{ccccc} 3&4&-2&-2&-2
%\\ \noalign{\medskip}4&3&-2&-2&-2\\ \noalign{\medskip}2&2&-2&
%-1&-1\\ \noalign{\medskip}2&2&-1&-2&-1\\ \noalign{\medskip}2&2&
%-1&-1&-2\end {array} \right] 
\end{align}
This matrix has characteristic polynomial $-\left(4-\lambda \right) \left( 1+\lambda \right)
^{4}$.  The spectral radius is the largest root, which is $4$.
Using (\ref{EQN:EQUAL_DEGREES}) we have that $\lambda_1(H) = \lambda_1(\H) = 4$.
\qed (Theorem \ref{THM:MAIN} ).

\begin{remark} 
There seems to be no convention on how to express a pullback on
cohomology in matrix notation.  We have written (\ref{EQN_MATRIX}) so
that $\mathcal{H}^*$ acts on column vectors.  However there are many other papers where
the transpose is used.  
\end{remark}

\section{Concluding questions and remarks}\label{SEC_QUESTIONS}

Computing the dynamical degrees of $H$ and gleaning the immediate dynamical consequences is really just the first step towards an understanding of it as a complex dynamical system.  There are many interesting avenues for further study and connections with other areas.  Among them, we list the following two:

\subsection{Complex proof of theorem on convex classes} One of the key theorems in Schwartz's monograph \cite{SCHWARTZ_PMM} is that (the projective equivalence class) of any convex pentagon has orbit under $H$ converging to (the projective equivalence class) of the regular pentagon.  He gives a computer assisted proof.  Can it be done using techniques from complex dynamics?  Can it be generalized to arbitrary $n$-gons?  A solution would probably require a suitable complex notion of ``convex''.

\subsection{Polarized dynamical system}
The eigenvector of $\H^{\ast}$ corresponding to $\lambda=4$ is 
\[-K_\X = 2\pi^{\ast}[\Lx]+2\pi^{\ast}[\Ly]-[E_1]-[E_2]-[E_3].\]
The referee pointed out the interesting fact that this is the anticanonical class of $\X$ (i.e. minus the canonical class of $X$); see \cite[p.146, p.187]{GH}.  

Recall that one can obtain $\CP^1 \times \CP^1$ by blowing up $\CP^2$ at two points and then blowing down the proper transform of the line between them.  Using this fact, one can recognize that our surface $\X$ is isomorphic to the blow up of $\CP^2$ at $[0:0:1],[0:1:0],[1:0:0]$ and $[1:1:1]$.  Such a surface is a del Pezzo surface  \cite{DOL,WIKI} and thus $-K_X$ is ample.
Because we have
\begin{align*}
\H^* \left(-K_\X \right) = 4 \left(-K_\X\right)
\end{align*}
with $-K_\X$ ample, $\H: \X \dashrightarrow \X$ is called a {\em polarized
dynamical system} \cite{KAW,ZHANG}.  Polarized dynamical systems
have been primarily of interest in arithmetic and algebraic dynamics and they
are usually studied for endomorphisms (i.e. rational mappings without
indeterminate points).  However, we mention the polarization for $\H$ in case there is a connection with arithmetic dynamics, despite the indeterminate points of $\H$, or in case it can be used for some other purpose.

\bibliographystyle{plain}
\bibliography{midpoint_map.bib}

\end{document}